\documentclass[12pt]{article}
\usepackage{amssymb,amsfonts,amsmath,amsthm,amscd,enumerate}
\usepackage{makeidx}
\usepackage[english]{babel}
\usepackage{color}
\usepackage{gensymb}
\usepackage{comment}
\usepackage{pst-node}
\usepackage{tikz-cd} 
\usepackage{soul}

\makeindex

\newtheorem{thm}{Theorem}[section]
\newtheorem{cor}[thm]{Corollary}
\theoremstyle{lem}
\newtheorem{lem}[thm]{Lemma}
\newtheorem{prop}[thm]{Proposition}
\newtheorem{defn}[thm]{Definition}
\theoremstyle{rem}
\newtheorem{rem}[thm]{Remark}
\newtheorem{exe}[thm]{Example}

\newcommand{\G}{\overline{G}}

\newcommand{\R}{\mathbb{R}}

\newcommand{\U}{\overline{U}}
\newcommand{\M}{\overline{M}}
\newcommand{\A}{\overline{h}}
\newcommand{\V}{\overline{V}}
\newcommand{\Gr}{\textbf{Gr}}

\newcommand{\af}{\alpha}
\newcommand{\m}{{}^{-1}}

\newcommand{\bt}{\beta}
\newcommand{\gra}{\textbf{Gr}(\alpha)}

\DeclareMathAlphabet{\mathpzc}{OT1}{pzc}{m}{it}

\def\<{{\langle}}
\def\>{{\rangle}}

\title{\textbf{Partial Groupoid Actions on Smooth Manifolds}}
\author{V\'{\i}ctor Mar\'{\i}n $^{\text{1}}$\\
   \small $^{\text{1}}$Departamento de Matem\'{a}ticas y Estad\'{i}stica, Universidad del Tolima, Ibagu\'{e}, Colombia\\
   \small e-mail: vemarinc@ut.edu.co\\
   H\'{e}ctor Pinedo $^{\text{2}}$\\
   \small $^{\text{2}}$ Escuela de Matem\'{a}ticas, Universidad Industrial de Santander, Bucaramanga, Colombia\\
   \small  e-mail: hpinedot@uis.edu.co\\
   José L. Vilca Rodríguez $^{\text{3}}$ \footnote{The second named-author was supported by   Fundação de Amparo à Pesquisa do estado de São Paulo  (FAPESP), process n°: 2023/14066-5, and
the third-named author was supported by FAPESP,  process No. 2019/08659-8, and PRPI da Universidade de São Paulo, process No. 22.1.09345.01.2. He also thanks the Department of Mathematics of Universidad del Tolima for its warm hospitality during his visit.}\\
   \small $^{\text{3}}$ Intituto de Matemática e Estatística, Universidade de São Paulo, São Paulo, Brazil\\
   \small e-mail: jvilca@ime.usp.br}
\date{\today}
\begin{document}
\maketitle
\begin{abstract}
\noindent Given a smooth partial action $\alpha$ of a Lie groupoid $G$ on a smooth manifold  $M,$ we provide necessary and sufficient conditions for  $\alpha$  to be globalizable with smooth globalization.  As an application, we provide results on the differentiable structure of orbit and stabilizer spaces induced by $\alpha,$ which leads to other criteria for its globalization in terms of its orbit maps in the case that $\alpha$ is free and transitive. Further, under the assumption that $\alpha$  is  free and proper, we prove that there exists exactly one differentiable structure on the quotient structure of the orbit equivalence space $M/G$ such that the quotient map $\pi:M\to M/G$ is a submersion. 
\end{abstract}
\-or\-bit equivalence relation.

\section{Introduction} 
The notion of the flow of a differentiable manifold likely emerged in the 1800s,  this concept constitutes a partial action of the additive group of real numbers~\cite[Example 1.2]{AB}. Flows of differentiable vector fields are crucial instances of local transformation groups, and the notion of a ``maximum local transformation group" described in R. Palais's memoir~\cite[p. 65]{Palais} essentially corresponds to the concept of a partial group action within a differentiable context. {The term {\em partial group action}  has been used in different but close senses. In our sense,  the concept of a partial group action on a $C^*$-algebra was gradually worked out in \cite{exel1994} and \cite{McClanahan} with the most general definition given in \cite{Exeltwisted}.
Since then, many applications of partial group actions  have been given, and recently  partial group actions on non-associative algebras, such as  Jordan, reductive or Lie,  have been studied in \cite{Cortes, DokS2, NYO, jose2022}. 
A natural question is whether a partial group action can be realized as restrictions of a corresponding collection of total
maps on some superspace. {To the best of our knowledge, this problem was ﬁrst addressed in the above-mentioned memoir by R. Palais for partial actions of Lie
groups on (non-necessarily Hausdorff) smooth manifolds (see~\cite[Theorem X, p. 72]{Palais}). Later, in the topological context, it was studied {independently} in  \cite{AB} and  \cite{KL}, where   the authors showed that for any continuous partial action  $\alpha$ of a topological group 
on a topological space $X,$  there is a topological space $Y$ and a continuous action $\beta$ of $G$ on $Y$ such that $X$ is a
subspace of $Y$ and $\alpha$ is the restriction of $\beta$ to $X.$ Such a space $Y$ is called a globalization of $X,$  it was also shown
that there is a minimal globalization $X_G$ called the enveloping space of $X.$ 

Partial group actions can be described in terms of premorphisms,  as explained in \cite{KL}, a partial action of a group $G$ on a set $X$ is a unital premorphism from $G$ to the inverse monoid $I(X),$ consisting of { all} bijections between subsets of $X.$ 
This perspective motivated the study of partial actions of other algebraic structures rather than groups, permitting the development of partial groupoid actions on sets, rings, topological spaces, and interactions between them.  In this new setting, the globalization problem has been considered and explored in different contexts, (see for instance \cite{ BP, BPP,  BPi, Gilbert,  MP, MP1,  Megre1, NY}), and as in the frame of partial group actions, it is observed that not every structural property of the ring/topological space endowed with a partial action,  is preserved by globalization.


Recently, Lie groupoids have been widespread in several areas of mathematics, and in recent years some of their “higher” versions have drawn much attention, particularly in the study of higher categorical structures, higher stacks, and higher gauge theory. Moreover the theory of Lie groupoids provides a convenient language in the development of the foundations of the theory of orbifolds and foliations \cite{M},  and as important examples of Lie groupoids we mention the so-called action groupoids induced by Lie group actions,  cover groupoids, submersion groupoids  and Lie groups. On the other hand, the theory of global groupoid actions on differential geometry appeared in  ~\cite{Dufour, Mackenzie, Mrcun},  and more recently in \cite{BNZ}. For an excellent overview of Lie groupoids, the interested reader may consult \cite{M}.

This work is a starting point for the study of partial actions of Lie groupoids on smooth manifolds and it is structured as follows. After the introduction, we present in   Section \ref{s2} the necessary background on groupoids, set-theoretical partial groupoid actions, and results on their globalizations that will be used in the rest of the paper. Moreover, given a groupoid $G$ we present the category   ${\bf Pact}(G)$ of partial actions of $G,$ at the end of this section we present in Theorem \ref{th1} a proof that the globalization of a transitive partial groupoid action is isomorphic in ${\bf Pact}(G)$ to certain left coset action, this result extends \cite[Theorem 2.6]{Choi} for the case of partial action of groups on sets, (see  Corollary \ref{ChoiLin}). 
Later, in Section \ref{s3} we introduce, in  Definition \ref{def_partialaction}, the notion of smooth partial action of a Lie groupoid $G$ on a smooth manifold $M,$ and study in  Lemma \ref{properties} and Propositions \ref{pro1} and \ref{differentiable_t}  topological and smooth properties of its enveloping space $M_G.$ In general,  the  space $M_G$  is not  Hausdorff, not even when $G$ is a  group, and in \cite[Proposition 1.2]{AB}, the author presented a characterization for the enveloping space to be Hausdorff, this idea that was naturally extended to the partial groupoid setting in \cite[Corollary 4.5]{MP}, in this sense, we give in Theorem \ref{th2}   necessary and sufficient conditions for  the smooth partial action to have  smooth globalization.  
In the first part of our last Section,   we deal with orbit maps and orbit equivalence spaces. Here we show, in Proposition  \ref{globalrank},  that any groupoid action has a constant rank,  a result that gives consequences on the manifold structure of the stabilizer and orbits of a smooth partial action. 
Later  we present, in Theorem \ref{thr1_smooth}, a criteria for the existence of a globalization of a free and  transitive smooth partial action. At the end of the work, we show in Theorem \ref{th3} and Corollary \ref{cf} that under the assumption that a graph open smooth partial action $\alpha$ of $G$ on $M$ to be free and proper, then there is exactly one differentiable structure on the quotient  $M/G$  such that the quotient map $\pi_G : M \to M/G$ is a submersion, provided that either $\alpha$ has a smooth globalization or the orbit map $\alpha^x:G^x\to M$ has a constant rank.

\section{Preliminary notions}\label{s2}
A groupoid is  defined as a small category where each morphism is invertible.  But also can be considered as an algebraic structure, which is defined by a set of axioms. For convenience, we adopt the algebraic axiomatic definition of a groupoid. Thus a  {\em groupoid} is a set $G$ endowed with a partial binary operation $(g,h)\mapsto gh$, i.e. whose domain is a subset of $G\times G$, which satisfy the following conditions for all $g,h,k\in G$: 
\begin{enumerate}
    \item $(gh)k$ exists if and only if $g(hk)$ exists, and in this case $(gh)k=g(hk);$
    \item $(gh)k$ exists if and only if $gh$ and $hk$ exist;
    \item for each $g\in G$ there exist unique elements $d(g),r(g)\in G$ such that $gd(g)$ and $r(g)g$ exist and $gd(g) = g = r(g)g;$
    \item for each $g\in G$ there exists an element $g^{-1}\in G$ such that $d(g)=g^{-1}g$ and $r(g) = gg^{-1}$.
\end{enumerate}

It follows that the element $g^{-1}$ is unique, $(g^{-1})^{-1}=g$ and $d(g^{-1})=r(g).$   Also,  the element $gh$ exists if and only if $h^{-1}g^{-1}$ exists, which,  in turn, is equivalent to $r(h)=d(g)$. We denote by $G^2$ the set of pairs $(g,h)\in G\times G$ such that $d(g)=r(h)$. An element $e\in G$ is  an {\em identity} of $G$ if $e=d(g)=r(g^{-1}),$ for some $g\in G$. The set of identities of $G$ is denoted by $G_0$. Moreover, for any $e\in G_0,$ the set $G_e:=r\m(e)\cap d\m(e)$  is  a group,  called the {\em isotropy
(or principal) group associated to $e.$} We recall the next:

\begin{defn}\label{ny}{\rm
A  (set-theoretic) \emph{partial action} $\alpha$ of a groupoid $G$ on a  set $X$ is a pair $\alpha=(X_g, \alpha_g)_{g\in G},$ where $\{X_e\}_{e\in G_0}$ is a family of pairwise disjoint sets,   for each $g\in G, X_g\subseteq  X_{r(g)}\subseteq X$  and $\alpha_g:X_{g^{-1}}\rightarrow X_g$ are bijections such that for every $(g,h)\in G^2$:
\begin{enumerate}
\item[(i)] $X=\bigsqcup_{e\in G_0}X_e$ and  $\alpha_e$ is the identity map ${\rm id}_{X_e}$ of $X_e,$ for all $e\in G_0,$
\item[(ii)] $\alpha_{g}(X_{g^{-1}}\cap X_{h})\subseteq X_{gh}$;
\item[(iii)] $\alpha_g(\alpha_h(x))=\alpha_{gh}(x)$, for every $x\in \alpha_h^{-1}(X_{g^{-1}}\cap X_h)$.
\end{enumerate}
We say that $\alpha$ is {\it global} if, $\alpha_{gh}=\alpha_g\circ \alpha_h, $ for any $(g,h)\in G^2.$ }
\end{defn}
If $\alpha$ is a partial action of $G$ on $X$ according to Definition \ref{ny}, then the condition $X=\bigsqcup\limits_{e\in G_0}X_e$ is equivalent to the existence of a surjective mapping $p:X\to G_0$
such that $X_e=p^{-1}(e).$ {T}he map $p$ is known as  the {\it anchor} or  the {\it moment map} {of $\alpha$,} and it is related to the notion of global groupoid action, in which the groupoid acts on a fibred set   (see for instance \cite{L}).  On the other hand, condition (iii) implies that $\alpha_{g^{-1}}=\alpha_g^{-1}$. Moreover, condition (ii) is equivalent to the equality $\alpha_g(X_{g^{-1}}\cap X_h)=X_g\cap X_{gh}$ for all $(h,g)\in G^2$.  Note that if  one defines the composition $\alpha_g\circ \alpha_h$ on $X_{h^{-1}}\cap X_{(gh)^{-1}}$, which is the largest possible domain on which it makes sense, then conditions (ii) and (iii) say that $\alpha_{gh}$ is an extension of $\alpha_g\circ \alpha_h,$ for any pair $(g,h)\in G^{2}.$ 

\begin{rem} {\rm The definition of a partial groupoid action on a set has been presented in a different way~\cite{Gilbert, GY}. The main distinction between their definition and ours is in condition~(i); where the authors in these works do not require that $X=\bigsqcup_{e\in G_0}X_e.$ 
 We prefer   Definition~\ref{ny} for the following reasons: 
\begin{enumerate}
\item  When $G$ is a group, our definition coincides with the classical notion of partial group actions on sets \cite[Definition 1.2]{E}. 
\item Condition~(i)  implies that $r(g)=r(h)$ whenever $X_{g}\cap X_h\neq \emptyset$ and this property allows us to extend certain properties of the partial actions to their globalization (as shown in Lemma~\ref{lema_globaltransi}). 
\item Definition~\ref{ny}  is consistent with the standard notion of (global) Lie groupoid action on differential geometry (see~\cite[Definition 1.6.1]{Mackenzie},~\cite[p. 125]{Mrcun}, \cite[Definition 7.1.20]{Dufour}), which arises naturally in the context of fiber bundles on manifolds, and with the notion of (global) action of categories~\cite[p. 332]{L}.
\end{enumerate}}
\end{rem}

\subsection{Induced partial actions and globalization}
Let $\beta=(Y_g, \beta_g)_{g\in G}$ be a partial action of $G$ on a set $Y$. A subset $X$ of $Y$ is called {\em $\beta$-invariant} if  $\beta_g(X\cap Y_{g^{-1}})\subseteq X,$ or equivalently, $\beta_g(X\cap Y_{g^{-1}})=X\cap Y_g,$
for all $g\in G$. Thus  setting $X_g=X\cap Y_g$ and $\alpha_g=\beta_g|_{X_{g^{-1}}}$ we obtain that  $\alpha=(X_g, \alpha_g)_{g\in G}$ is a partial action of $G$ on $X$. 
Now, assume that $X$ is an arbitrary subset of $Y$. Let 
\begin{equation}\label{alfares} X_g=X\cap \beta_g(X\cap Y_{g^{-1}})\,\,\,{\rm  and }\,\,\,\alpha_g=\beta_g|_{X_{g^{-1}}},\end{equation} it is not difficult to see that  $\alpha=(X_g, \alpha_g)_{g\in G}$  is a partial action of $G$ on $X$, which is called the  {\em restriction } of $\beta$ to $X$.  We proceed with the next.

\begin{prop}\label{rest} {\rm Let $\beta=(Y_g, \beta_g)_{g\in G}$ be a global  action of $G$ on a set $Y$  and $X\subseteq Y$.  Consider  $\overline{Y}_g=\bigcup\limits_{r(h)=r(g)}\beta_h(X_{d(h)}),$ where   $X_g$ is defined by \eqref{alfares}, for each $g\in G$. Then, the set $\overline{Y}=\bigsqcup\limits_{e\in G_0} \overline{Y}_e$ is the  smallest  $\beta$-invariant subset of $Y$ such that $X\subseteq \overline{Y}$ {and $\overline{Y}_e\subseteq Y_e$}.  Moreover,  $\beta$ induces a global action on $\overline{Y}$, whose restriction to $X$ coincides with the restriction of $\beta$ to $X$. }
\end{prop}
\begin{proof}It is clear that $X\subseteq \overline{Y}\subseteq Y.$ To check that $\overline{Y}$ is $\beta$-invariant, take $g\in G$ then $Y_{g^{-1}} \cap \overline{Y}=Y_{d(g)} \cap \overline{Y}=\overline{Y}_{d(g)},$ and 
\begin{align*} \beta_{g}(Y_{g^{-1}} \cap \overline{Y})=\beta_g(\overline{Y}_{d(g)})=\bigcup_{r(h)=d(g)}\beta_{gh}(X_{d(gh)})=\bigcup\limits_{r(l)=r(g)}\beta_{l}(X_{d(l)})=\overline{Y}_{r(g)}\subseteq \overline{Y},
\end{align*} 
since any $l\in G$ with $r(l)=r(g)$ can be written as $gh$ with $h=g^{-1}l$. This shows that $\overline{Y}$ is $\beta$-invariant. Further,  if $S\subseteq Y$ is a $\beta$-invariant subset of $Y$ such that  $X\subseteq S$,  we have that 
$$\beta_h(X_{d(h)})=\beta_h(X\cap Y_{h^{-1}})\subseteq\beta_h(S\cap Y_{h^{-1}})=S\cap Y_h;$$ and then $\overline{Y}_g=\bigcup\limits_{r(h)=r(g)}\beta_h(X_{d(h)})\subseteq S\cap Y_g$. Hence $\overline{Y}=\bigsqcup\limits_{e\in G_0} \overline{Y}_e\subseteq \bigsqcup\limits_{e\in G_0}(S\cap Y_e)=S$, as required.
\end{proof}
{ Given  Proposition \ref{rest}, it is natural to ask if every partial groupoid action on a set $X$ can be obtained as a restriction to $X$ of a global action on a set $Y \supseteq X$, such that $Y$ is minimal like $\overline{Y}$ above.} This problem is referred to as the globalization problem and has been treated in several contexts. To state it precisely, we  give the next.

\begin{defn}\label{mor}{\rm Let  $\alpha =(X_g, \alpha_g )_{g\in G}$ and $\widetilde\alpha =(\widetilde X_g, \widetilde\alpha_g)_{g\in G}$ be  partial actions of $G$ on the sets  $X$ and $\widetilde X$, respectively. A {\em homomorphism or ($G$-map)} between  $\alpha$ and $\widetilde\alpha$ is a map  $\varphi:X\to \widetilde X$ such that:
\begin{enumerate}
\item [(i)]  For every $g\in G$ and $x\in X_g, \varphi(x)\in \widetilde X_g ;$ 
\item [(ii)]  $\varphi\circ\alpha_g=\widetilde\alpha_g\circ\varphi$ on $X_{g^{-1}}.$
\end{enumerate} An  invertible homomorphism is called  an {\em isomorphism} {\em (or equivalence)}, in this case we say that $\alpha$ and $\tilde \alpha$ are equivalent.}
\end{defn} 
For instance, if $\beta$ is a global action of $G$ on a set $Y$, $X$ is a subset of $Y$, and $\alpha$ is a restriction of $\beta$ to $X$, then the inclusion $X\to Y$ is a $G$-map between $\alpha$ and $\beta$.  Moreover,  using Definition \ref{mor} one obtains the category ${\bf Pact}(G)$ of partial actions of $G,$ whose objects are the partial actions and morphism described above. Several structural properties of this category have been presented in \cite{MP1} and \cite{NY}. %

Let $\alpha$ be a partial action of $G$ on $X$. A global action $\beta$ of $G$ on a set $Y$ is a {\em globalization} for $\alpha$ if there exists a   monomorphism $i: X\to Y$ between $\alpha$ and $\beta$ such that the $\beta$-orbit of $i(X)$, i.e. the smallest $\beta$-invariant subset of $Y$ containing $i(X)$, coincides with $Y$. More explicitly,  in view of Proposition \ref{rest}, we have the next.

\begin{defn}\label{defglobaSM1}{\rm  A global action
$\bt=(Y_g, \bt_g)_{g\in G}$ of  a groupoid $G$ on a set $Y$ is
a {\it globa\-lization} of a partial action
$\af=(X_g, \alpha_g )_{g\in G}$ of $G$ on $X,$ if there exists an injective map $i:X \to Y$ such
that:
\begin{enumerate}\renewcommand{\theenumi}{\roman{enumi}}   \renewcommand{\labelenumi}{(\theenumi)}

\item $i(X_{g})=i(X) \cap \beta_g(i(X)\cap Y_{g^{-1}})$;
\item $\bt_g(i(x))=i(\af_g(x))$, for all $x\in X_{g^{-1}}$;
\item $Y_g=\bigcup\limits_{r(h)=r(g)}\bt_h(i(X_{d(h)}))$.
\end{enumerate}}
\end{defn}
It is known that a globalization of a partial groupoid action on a set always exists and it is unique up to isomorphism~\cite[Theorem 4]{NY}. Moreover, 
it follows from Definition~\ref{defglobaSM1} that if the action $\beta$ of $G$ on $Y$ is a globalization of the partial action $\alpha$ of $G$ on $X$, then the map $i:X\to i(X)$ is an isomorphism in ${\bf Pact}(G)$ between $\alpha$ and the restriction of $\beta$ to $i(X).$ Therefore we can assume, without loss of generality,  that $X$ is a subset of $Y$ and that the embedding $i: X\to Y$ of Definition~\ref{defglobaSM1} is given by the inclusion. 

\subsection{Transitive partial groupoid actions}
Let $\af=(X_g, \alpha_g )_{g\in G}$  be a   partial action of  a groupoid $G$ on a set $X$. In $X$ we define the following relation: 
\begin{equation}\label{orbig}
x\sim y\text{ if and only if  there is }  g\in G\text{ such that }x\in X_{g^{-1}} \text{ and }y=\alpha_g(x).\end{equation}
 It is clear that $\sim$ is reflexive and symmetric. To show that $\sim$ is an equivalence relation,  consider $x,y,z\in X$ so that $x\sim y$ and $y\sim z$. Then, there exist $g,h\in G$ so that $x\in X_{g^{-1}}, y\in X_{h^{-1}}$ and $y=\alpha_g(x), z=\alpha_h(y)$. Thus $y\in X_{h^{-1}}\cap X_g$, which implies $d(h)=r(g)$. Hence, 
$$y=\alpha_g(x)\in X_{h^{-1}}\cap X_g=\alpha_g(X_{g^{-1}}\cap X_{(hg)^{-1}}),$$
yielding $x\in X_{g^{-1}}\cap X_{(hg)^{-1}}$. Therefore $z=\alpha_h(y)=\alpha_h\alpha_g(x)=\alpha_{hg}(x),$ and  $x\sim z$, as required. The relation $\sim$ is called the {\it orbit equivalence relation},  the equivalence class of an element $x$ in $X$ is called
 the  {\em orbit} of $x$.

\begin{rem} {\rm If we remove condition (i) in  Definition \ref{ny} of partial groupoid action, then the orbit relation \eqref{orbig} is not necessarily transitive. Indeed, consider the groupoid $G=\{g,g\m, h,h\m, e,f\},$ where $e=d(g)=r(g)$ and $f=d(h)=r(h).$ Then $G$ acts partially (in the sense of \cite{Gilbert, GY}) on the set $X=\{x_1,x_2,x_3\}$ by 
\begin{itemize}
\item $X_{g\m}=\{x_1\}, X_{g}=\{x_2\}, X_{e}=\{x_1, x_2\}, \alpha_{g}(x_1)=x_2, \alpha_{g\m}=\alpha\m_{g}, \alpha_{e}={\rm id}_{X_e};$
\item $X_{h\m}=\{x_2\}, X_{h}=\{x_3\}, X_{f}=\{x_2, x_3\}, {\alpha_{h}}(x_2)=x_3, {\alpha_{h\m}}=\alpha\m_{h}, \alpha_{f}={\rm id}_{X_f}.$
\end{itemize}
Then $x_1\sim x_2, x_2\sim x_3$ but $x_1\nsim x_3.$}
\end{rem}

 Let $G^x=\{g\in G\mid x\in X_{g^{-1}}\}$, then 
 \begin{equation}\label{or}G^x\cdot x:=\{\alpha_g(x)\in X\mid g\in G^x\},\end{equation} 
 coincides with the orbit of $x$.  Moreover,   
\begin{equation}\label{om}\alpha^{x}:G^{x}\ni g\mapsto \alpha_g(x)\in X\end{equation} is called  the  {\em orbit map}.  The set $${\rm stab}_\alpha(x)=\{g\in G\mid x\in X_{g^{-1}} \mbox{ and }\alpha_g(x)=x\}= (\alpha^x)^{-1}(x),$$
is referred as the {\em  stabilizer of $x$}, and this is a subgroup of the principal group $G_{p(x)},$ where $p$ is the anchor map  of $\alpha$.

\begin{defn}{\rm
A partial action $\af=(X_g, \alpha_g )_{g\in G}$ of $G$ on $X$ is called: 
\begin{enumerate}
    \item[(i)] {\em transitive}, if there exists $x\in X$ such that $G^x\cdot x=X$; and 
    \item[(ii)] {\em free}, if ${\rm stab}_\alpha(x)=\{p(x)\}$ for all $x\in X$.
\end{enumerate}}
\end{defn}

One note that if $\alpha=(X_g, \alpha_g)_{g\in G}$ is a partial action of $G$ on $X$, the restriction  $\alpha^{(e)}=(X_g, \alpha_g)_{g\in G_e}$ is a partial action of the isotropy group $G_e$ on the set $X_e,$ for each $e\in G_0$.  Moreover, it is easy to see that if $\alpha$ is transitive, then  $\alpha^{(e)}$ is also transitive. Also, given $x\in X_e$ we have that ${\rm stab}_\alpha(x)\cap G_e={\rm stab}_{\alpha^{(e)}}(x)$. So if $\alpha$ is free,   the induced partial action $\alpha^{(e)}$ is free.

Now we show that the globalization of a transitive (resp. free) partial action is transitive (resp. free). This extends a result of Choi and Lim \cite[Proposition 2.4]{Choi} for transitive partial group actions to the groupoid setting.

\begin{lem}\label{lema_globaltransi} {\rm Let  $\af=(x_g, \alpha_g )_{g\in G}$  be a  partial action of  groupoid  $G$ on  a set $X$. Suppose that the global action  
$\bt=(Y_g, \bt_g)_{g\in G}$ of $G$ on $Y$ is a globalization for $\alpha$. Then, the following assertions hold:
\begin{enumerate}
    \item[(1)] $\alpha$ is transitive, if and only if, $\beta$ is transitive;
    \item[(2)] $\alpha$ is free, if and only if, $\beta$ is free.
\end{enumerate}}
\end{lem}

\begin{proof}
As $\beta$ is a globalization for $\alpha$ we have 
$$Y_g=\bigcup\limits_{r(h)=r(g)}\bt_h(X_{d(h)})\quad \mbox{  \,\,\,  and  \,\,\,  } \quad Y=\bigsqcup_{e\in G_0}Y_e.$$
(1) First we suppose $\alpha$ is transitive. Let $x,y\in Y.$ Then, there exist $h,k\in G$ such that $x=\beta_h(\widetilde x), y=\beta_k(\widetilde y)$ for some $\widetilde x\in X_{d(h)}, \widetilde y\in X_{d(k)}$. As $\alpha$ is transitive, there is $g\in G$ such that $\widetilde x\in X_{g^{-1}}$ and $\widetilde y=\alpha_g(\widetilde x)$, and so $\widetilde x\in X_{g^{-1}}\cap X_{d(h)}$, yielding $d(h)=r(g)$. On the other hand, $\widetilde y\in X_g\cap X_{d(k)}$ implies $r(g)=d(k)$. Hence  $y=\beta_k(\alpha_g(\widetilde x))=\beta_k\circ \beta_g\circ\beta_{h^{-1}}(x)=\beta_{kgh^{-1}}(x)$, showing that $\beta$ is transitive. 
Conversely, suppose that $\beta$ is transitive. Let $x,y\in X\subseteq Y$. Then, there exists $g\in G$ such that $x\in Y_{g^{-1}}$ and $y=\beta_g(x)$. As $y\in X \cap Y_{g}$, we have that $x=\beta_{g^{-1}}(y)\in X\cap \beta_{g^{-1}}(X\cap Y_g)=X_{g^{-1}}$. Therefore $y=\alpha_g(x)$, and so $\alpha$ is transitive, as required.

(2) Suppose that $\alpha$ is free. Let $p:Y\to G_0$ be the anchor map of $\beta$ and let $y\in Y$. As $\beta$ is a globalization for $\alpha,$ then it follows by (iii) in Definition \ref{defglobaSM1} that  there exists $h\in G$ and $x\in X_{d(h)}$ such that $y=\beta_h(x)$. If $g\in{\rm stab}_\beta(y)$, then $y\in Y_{g^{-1}}$ and $\beta_g(y)=y$. Thus $y\in Y_{g^{-1}}\cap Y_g\cap Y_h$,  $d(g)=r(g)=r(h)$, and 
$$\beta_{gh}(x)=\beta_g\beta_h(x)=\beta_g(y)=y=\beta_h(x),$$ yielding $\beta_{h^{-1}gh}(x)=x$. The last equality implies that $x\in X\cap Y_{h^{-1}g^{-1}h}$ and 
$$x=\beta_{h^{-1}gh}(x)\in X\cap \beta_{h^{-1}gh}(X\cap Y_{h^{-1}g^{-1}h})=X_{h^{-1}gh}$$ hence $x=\alpha_{h^{-1}gh}(x)$. Since $\alpha$ is free and $x\in X_{d(h)}$ we have that $h^{-1}gh=p(x)=d(h)$, which implies $g=hh^{-1}=r(h)=d(g)=r(g)=p(y).$ We conclude that ${\rm stab}_\beta(y)=\{p(y)\}$, as required. Now suppose that $\beta$ is free and let $x\in X$. Then, it follows from the equality ${\rm stab}_\beta(x)={\rm stab}_\alpha(x)$ that $\alpha$ is free.
\end{proof}

 It is a well-known fact that if $G$ is a group, every transitive partial  action of $G$ on a set $X$ is isomorphic to the left coset action of $G$ on $G/{\rm stab}(x)$ for any $x\in X$. To obtain an extension of this result to the groupoid setting, we introduce a global action on an adequate set.

Let $\af=(X_g, \alpha_g )_{g\in G}$  be a partial action of $G$ on $X$ and  $x\in X$. Let us denote by $H^x=\{h\in G\mid d(h)=p(x)\} =d^{-1}(p(x))$. In $H^x$ we define the following equivalence relation: 
\begin{equation*}\label{equiv}
h_1\sim h_2\,\,\,\text{ if and only if }\,\,\, (h_2^{-1},h_1)\in G^2 \,\,\,\text{ and}\,\,\,  h_2^{-1}h_1\in {\rm stab}_\alpha(x).
\end{equation*} 
We denote the equivalence class of  $h\in H^x$ by $\A$ and  set 
\begin{equation}\label{barhx}\overline{H}^x=\{\A\mid h\in H^x\}.\end{equation} One can see that if $\alpha$ is free, then $\A=\{h\}$, and hence $\overline{H}^x$ equals $H^x.$  Let us define the mapping $t:\overline{H}^x\ni \A\mapsto r(h)\in G_0$ {and $G\times _t\overline{H}^x=\{(g,\overline{h})\in G\times {H}^x\mid t(\overline{h})=d(g)\}$,} it  is easy to see that $t$ is well-defined  and
\begin{equation}\label{delta}\delta: G\times _t\overline{H}^x\ni (g,\A) \mapsto \overline{gh}\in \overline{H}^x,\end{equation} is a well-defined global action of $G$ on $\overline{H}^x$.  Moreover, for each $g\in G$ we have that $$\overline{H}^x_g=\{\A\in \overline{H}^x\mid (g^{-1}, h)\in G\times _t\G\}= \{\A\in \overline{H}^x\mid r(h)=r(g)\}.$$ 
Note that if $\alpha$ is free, then  $\delta$ is a global action of $G$ on $H^x$.
\begin{thm}\label{th1}

{\rm Let $\af=(X_g, \alpha_g )_{g\in G}$ be a transitive partial action of $G$ on a set $X$ and let  $x\in X$. Suppose that the global action  
$\bt=(Y_g, \bt_g)_{g\in G}$ of $G$ on $Y$ is a globalization for $\alpha$. Then, the mapping $\varphi: \overline{H}^x\ni \A \mapsto \beta_h(x)\in Y$, is an isomorphism between the actions $\delta$  and $\beta.$}
\end{thm}
\begin{proof}
First, we see that  $\varphi$ is well-defined. Indeed,  the conditions  $(h_2^{-1},h_1)\in G^2$ and $h_2^{-1}h_1\in {\rm stab}_G(x)$ imply that $x=\beta_{h_2^{-1}h_1}(x)$, yielding $\beta_{h_1}(x)=\beta_{h_2}(x)$. To see that $\varphi$ is injective let $h_1,h_2\in H$ such that $\beta_{h_1}(x)=\beta_{h_2}(x)$. Then $r(h_1)=r(h_2)$ and $$x=\beta_{h_1^{-1}}\beta_{h_2}(x)=\beta_{h_1^{-1}h_2}(x)\in X\cap Y_{h_1^{-1}h_2},$$ 
and thus $$x=\beta_{h_2^{-1}h_1}(x)\in X\cap Y_{h_1^{-1}h_2}\cap \beta_{h_2^{-1}h_1}(X\cap Y_{h_1^{-1}h_2})=X_{h_2^{-1}h_1}.$$
Therefore $x=\alpha_{h_2^{-1}h_1}(x)$, which implies that $h_2^{-1}h_1\in {\rm stab}_\alpha(x)$, yielding $\overline{h_1}= \overline {h_2}$, as required. To see that $\varphi$ is onto, note that, by Lemma~\ref{lema_globaltransi}, $\beta$ is transitive. Hence,  given $y\in Y$ there exists $h\in G$ such that $x\in Y_{h^{-1}}$ and $y=\beta_h(x)$, showing that $\varphi$ is onto.

Finally, we {show} that $\varphi$ is a morphism in  ${\bf Pact}(G)$. Let $\A\in \overline{H}^x_{g^{-1}}$. Then $d(g)=r(h)$, and {thus} 
$\varphi(\delta_g(\A))=\varphi(\overline{gh})=\beta_{gh}(x)=\beta_g(\beta_h(x))=\beta_g(\varphi(\A)),$
as required.
\end{proof}

Note that if $G$ is a group then for any element $x$ in $X$ we have that $H^x=G$, $\overline{H}^x=G/{\rm stab_\alpha(x)}$ and $\delta$ is the left coset (global) action of $G$ on $G/{\rm stab}_\alpha(x)$. So, we recover the following result which, in turn, is a generalization of a classical result on group actions.

\begin{cor}{\rm(\cite[Theorem 2.6]{Choi})}\label{ChoiLin}
{\rm Suppose that $\alpha$ is a transitive partial action of a group  $G$ on $X$. Then, any globalization of $\alpha$ is isomorphic to the left coset action of $G$ on $G/{\rm stab}_\alpha(x)$ for any $x\in X$.}
\end{cor}
 Let us denote by $\delta^{p(x)}$ the left coset action of the isotropy group $G_{p(x)}$ on the set $G_{p(x)}/{\rm stab}_{\delta^{p(x)}}(x)$. Then, as an immediate consequence, we finish this section with the next.

\begin{cor}\label{prop_transi}{\rm Let  $\alpha$ be a transitive partial action of a group  $G$ on $X$. If $\beta$ is a globa\-lization for $\alpha$, then for each $x\in X$ the actions $\beta^{(p(x))}$ and $\delta^{p(x)}$ are isomorphic.}
\end{cor}
\section{Globalization of partial Lie groupoid actions on manifolds}\label{s3}

In this section, we study partial actions in a differential context. More specifically, we treat  the globalization problem for a partial group action of a Lie groupoid on a smooth manifold and present some results related to this. 

In what follows the term smooth manifold means a Hausdorff, second countable, locally Euclidean topological space endowed with a differential structure. For convenience, all  manifolds considered in this work  are $C^\infty$-differentiable.  Following the terminology in~\cite{Lee} we say that a subset $S$ of a smooth manifold M is an {\em embedded submanifold} of $M$ if it is a smooth manifold with the subspace topology,  and we say that { $S$ is an {\em (immersed) submanifold} of $M$ if it  is   a smooth manifold endowed with a certain  topology  such that the inclusion $S\hookrightarrow M$ is an immersion.}

We recall that a groupoid $G$ is said to be a \emph {Lie groupoid} if it is endowed with a structure of smooth manifold compatible with its algebraic structure, that is:
\begin{enumerate}
\item $G_0$ is an embedded closed submanifold of $G;$
\item the inversion $ {\rm inv}: G\to G$ and the multiplication map ${m}: G^{2}\to G$ are smooth; and
\item the maps $d,r: G\to G_0$ are surjective submersions.
\end{enumerate}
It is a well-known fact that the inversion map ${\rm inv}: G\to G$ is a diffeomorphism (see~\cite[Proposition 1.1.5]{Mackenzie}).
Following~\cite{NY}, we say that a Lie groupoid $G$ is {\it star open} if $d\m(e)$ is open in $G,$ for any $e\in G_0,$ for instance every Lie group is a star open Lie groupoid.   The next result characterizes star open Lie groupoids.

\begin{prop}\label{soc} {\rm Let $G$ be a Lie groupoid. The following assertions are equivalent.
\begin{enumerate}
\item[(1)] $G$ is star open.
\item[(2)]  $r^{-1}(e)$ is open in $G,$ for any $e\in G_0.$
\item[(3)]  $G_0$ is discrete.
\end{enumerate}}
\end{prop}

\begin{proof} Since the inversion map ${\rm inv}: G\to G$ is a diffeomorphism,  it follows that $G$ is star open, if and only if, $r^{-1}(e)={\rm inv}^{-1}(d^{-1}(e))$ is open in $G,$  for any $e\in G_0.$ This shows (1) $\Leftrightarrow$ (2).  Now for (1) $\Rightarrow$ (3), it follows from~\cite[Corollary 5.13]{Lee} that each $d^{-1}(e), e \in G_0,$ is an embedded submanifold of codimension $\dim G_0$, but as $d^{-1}(e)$ is open in $G$ its dimension coincides with $\dim G$, yielding $\dim G_0=0$, as required. Finally, for  (3) $\Rightarrow$ (1)  assume that $G_0$ is discrete, then, by continuity, $d^{-1}(e)$ is open in $G$ for any $e\in G_0$.
\end{proof}


\begin{defn}\label{def_partialaction}{\rm
 Let $M$ be a smooth manifold. A {\em smooth partial action } of a Lie groupoid $G$ on $M$ is a set-theoretic partial action 
$\alpha =(M_g, \alpha_g)_{g\in G}$ of the abstract groupoid $G$ on the underlying set $M$, such that:
\begin{enumerate}
   \item[(i)] For each  $g\in G$, $M_g$ is a submanifold of $M$ and $\alpha_g:M_{g^{-1}}\to M_g$ is a diffeomorphism.
    \item[(ii)] The anchor map $p:M\to G_0$ is smooth.
    \item[(iii)] The set $\Gamma_\alpha=\{(g,x)\in G\times M\mid x\in M_{g^{-1}}\}$ is a submanifold of $G\times M$, and the map (also called $\alpha$) $\alpha:\Gamma_\alpha\to M$, defined by $\alpha(g,x)=\alpha_g(x)$, is smooth. 
\end{enumerate}
}
\end{defn}

\noindent Given a set $X$ and  a surjection $s: X\to G_0$, we  set
$G\times _s X=\{(g,x)\in G\times X\mid s(x)=d(g)\},$ which coincides with the fiber product of $s$ and $d.$ Note that if $\alpha$ is global action then $\Gamma_\alpha=G\times_p M$ is automatically a submanifold of $G\times M$ {(see \cite[Example 13.24]{Cra}).}

\begin{defn}
{\rm A smooth partial action $\alpha =(M_g, \alpha_g)_{g\in G}$ of $G$ on $M$ is called {\em graph open}\footnote{This term  comes from the topological context language \cite[Definition 4.1 (iii)]{MP1}.} if $\Gamma_{\alpha}$ is an open subset of $G\times M$.}
\end{defn}

\begin{exe} {\rm Graph open partial actions of Lie groups and restrictions of  graph open partial actions of Lie groupoids, both smooth, on open  submanifolds are  graph open partial actions of Lie groupoids. }\end{exe}

If $\alpha$ is a graph open partial action of $G$ on $M$,  
then  follows from  condition~{(iii)} of Definition~\ref{def_partialaction} that for each $g\in G$, the set $M_{g}$ is an open subset of $M,$  
 $G^x=\{g\in G\mid  (g,x)\in \Gamma_\alpha \}$ is an  open submanifold of $G,$ and  $\alpha^{x}:G^x\ni g\mapsto \alpha(g,x)$ is smooth, for any $x\in M.$ Finally,  a calculation shows that 
 \begin{equation}\label{edife}
 d\alpha(g_0,x_0)(v,w)=d\alpha^{x_0}(g_0)v+d\alpha_{g_0}(x_0)w, 
 \end{equation}
 for  $(g_0, x_0)\in\Gamma_\alpha$ and $(v,w)\in T_{g_0}G\times T_{x_0}M$, which implies that $\alpha$ is a submersion,  as $\alpha_{g_0}$ is a diffeomorphism.

\begin{defn}\label{defglobaSM} {\rm A global action
$\bt=(N_g, \bt_g)_{g\in G}$ of $G$ on a smooth manifold $N$ is
a {\it globalization} of a smooth partial action 
$\af=(M_g, \alpha_g)_{g\in G}$  if, 
 there exists a smooth map $\iota:M\to N$ such
that $\iota(M)$ is an open subset of $N,$ $\iota:M\to N$ is a diffeomorphism onto its image; and (i)-(iii) in Definition \ref{defglobaSM1} are satisfied. If $\alpha$ has a globalization we say that $\alpha$ is {\em globalizable}.}
\end{defn}

In the topological setting, the construction of a globalization for a partial groupoid action is presented in \cite[Theorem 5]{NY}, for the reader's convenience we recall their construction.

 Let $\alpha=(M_g,\alpha_g)_{g\in G}$ be a graph open smooth partial action of $G$ on $M$. Let us consider the set  $$\M=G\times_p M=\{(g,x)\in G\times M\mid x\in M_{d(g)}\}=\{(g,x)\in G\times M\mid p(x)=d(g)\}$$ with the topology induced by the product topology, and define the following equivalence relation on $\M:$
\begin{equation}\label{equivr}(g,x)\sim (h,y)\,\,\,\,\text{ if and only if }\,\,\,\,(g^{-1},h)\in G^2, x\in M_{g^{-1}h}\,\,\text{ and }\,\,y=\alpha_{h^{-1}g}(x).
\end{equation} As usual, we denote by $[g,x]$ the equivalence class of $(g,x)\in \M$ and by $M_G$  the quotient  space $\M/\sim$ endowed with the quotient topology. 
\begin{rem}{\rm It is not difficult to see that, the equivalence relation in \eqref{equivr}  is the same as the one given by Nystedt in~\cite{NY}. Indeed, the condition $M=\bigsqcup\limits_{e\in G_0} M_e$ in Definition~\ref{ny} makes the second part of
\cite[Definition 12]{NY}  irrelevant.}
\end{rem}
Set
\begin{equation}\label{eq_anchor}
t: M_G\ni [h,x]\to r(h)\in G_0.
\end{equation}
Then, $t$ is well-defined and 
\begin{equation}\label{eq_globalization}
\beta:G\times_t M_G \ni (g,[h,x])\mapsto[gh,x]\in M_G, 
\end{equation}
is a well-defined set-theoretic global action of $G$ on $M_G$ (see~\cite[Theorem 4]{NY}), in which $$(M_G)_g=\{[h,x]\in M_G\mid r(g)=r(h)\}$$ and $\beta_g: (M_G)_{g^{-1}}\ni [h,x]\mapsto [gh,m]\in  (M_G)_g$ is a bijection. \\%

\begin{lem}\label{properties}{\rm 
Suppose that $G$ is a star open  Lie groupoid, $M$ a smooth manifold, and  $\alpha=(M_g, \alpha_g)_{g\in G}$ a  graph open smooth partial action of $G$ on $M.$ Then the following assertions  hold: 
\begin{enumerate}
\item[(1)] The natural projection $\pi:\M\to M_G$ is an open map. 
\item[(2)] The  action $\beta: G\times_t M_G\to M_G$ is continuous. Thus, each $\beta_g:(M_G)_{g^{-1}}\to (M_G)_g$ is a  homeomorphism.  
\item[(3)] The map $\iota: M\ni x\mapsto [p(x), x]\in M_G$ is injective, continuous and open. Therefore, it is a  homeomorphism onto $\iota(M)$. 
\item[(4)] For an open subset $U$ of $M$ we have that $\iota(U)\cap (M_G)_{d(g)}=\iota(U\cap M_{d(g)}).$ 

\end{enumerate}  

}
\end{lem}
\begin{proof}
(1) 
{This is a consequence of  \cite[Proposition 4.4 (ii)]{MP1}.}

(2) We consider the set  $X=\{(g,(h,x))\in G\times \M\mid (g,h)\in G^2\}$, and define $\theta: X\ni (g, (h,x))\mapsto (gh,x)\in  \M,$ we shall check that $\theta$ is continuous. Let $U\subseteq G$ and $V\subseteq M$ be open sets, since $\theta^{-1}(U\times V\cap \M)=(m^{-1}(U)\times V)\cap X$  and  the multiplication map $m:G^2\to G$ is continuous,  then  $\theta^{-1}(U\times V\cap \M)$ is open, and therefore $\theta$ is continuous.  Now, in $X$ we define the following equivalence relation:
\begin{center} $(g,(h,x))\equiv (g',(h',x'))$ if and only if $g=g'$ and $[h,x]=[h',x']$. \end{center}
Denote by {$\overline{(g,(h,x))}$ the equivalence class of $(g,(h,x))\in X$}. Put $Y=X/\equiv $ and let  $\overline{\pi}:X\to Y$ be the natural projection,  and endow $Y$ with the quotient topology. Let $$\phi:Y\ni {\overline{(g,(h,x))}}\mapsto (g,[h,x])\in G\times_t{M_G}.$$ Since $\phi\circ \overline{\pi}=id\times\pi$,  we have that  $\phi$ is continuous. Now we check that $\phi$ is open, for this  let $O$ be an open subset of $Y,$ then there exist open subset $U\subseteq G$ and $V\subseteq \M$ such that $\overline{\pi}^{-1}(O)=(U\times V)\cap X$, and it is not difficult to see that $\phi(O)=U\times \pi(V)$. Thus $\phi$ is an open map, {because $\pi$ is open}, and hence a homeomorphism, as it is bijective. Finally, from the equality  $\beta\circ \phi\circ \overline{\pi}=\pi\circ \theta$, we get that  $\beta$ is  continuous.  

(3) The fact that $\iota$ is injective, follows from the disjoint union $M=\bigsqcup\limits_{e\in G_0}p^{-1}(e).$ Furthermore, $\iota$ is continuous since $\pi$ and $p$ are continuous. Finally,~\cite[Proposition 38]{NY} shows that $\iota$ is an open map. 

 (4) For $g\in G$ we have
\begin{align*}
    \iota(U)\cap (M_G)_{d(g)}&=\{[p(y),y]\mid y\in U\mbox{ and }p(y)=d(g)\}\\
    &=\{[p(y),y]\mid y\in M_{d(g)}\cap U\}=\iota(U\cap M_{d(g)}),
\end{align*}
as desired.
\end{proof}


Our next goal is to obtain a {graph open} smooth globalization for  {a graph open smooth partial action} $\alpha$.  As a first step, we notice that $M_G$ is second countable, as  $\pi:\M\to M_G$ is a surjective open map, and $G$ and $M$ are second countable. As a second step, we construct a differential structure for $M_G$.   At this point, we assume that  $G$ is star open.
For $[g,x]\in M_G$ we fix a representative $(g,x)\in \M$ of it; and for this representative we choose a chart $\varphi: U\to \R^n$ of $M$ such that $x\in U$. Then, by 4) in  Lemma~\ref{properties},  the set  $\U_{g,x}=\beta_g \iota(U\cap M_{d(g)})$ is  open in $(M_G)_g$, and therefore  in $M_G$,  also $[g,x]\in \U_{g,x}$.
Let $\varphi_{g,x}:\U_{g,x}\to \R^n$ be the composition 
$ \U_{g,x}\stackrel{\beta_{g^{-1}}}{\longrightarrow}\iota(U\cap M_{d(g)})\stackrel{\iota^{-1}}{\longrightarrow}U\cap M_{d(g)}\stackrel{\varphi}{\longrightarrow}\R^n,$
that is, $\varphi_{g,x}([g,y])=\varphi(y)$ for all $[g,y]\in \U_{g,x}$.
Since $\iota$ is a homeomorphism onto its image and $\beta_g$ is a homeomorphism, it follows that $\varphi_{g,x}$ is a well-defined homeomorphism from $\U_{g,x}$ onto an open subset of $\R^n$. 
\begin{prop}\label{pro1}{\rm 
Let $(g,x)$ and $(h,y)$ be two arbitrary points of $ G\times_p M$, and let $(\varphi,U)$ and $(\psi,V)$ be two compatible charts of $M$ such that $x\in U$ and $y\in V$. Then, the charts $\varphi_{g,x}$ and $\psi_{h,y}$, defined above, are compatible. Consequently, the maximal atlas containing all the $\varphi_{g,x}$, with $[g,x]\in M_G$, does not depend of the choice of the representative of $[g,x]\in M_G.$}
\end{prop}
\begin{proof}
Suppose that $\U_{g,x}\cap \V_{h,y}\neq \emptyset$. Let us see that the change of coordinates $$\psi_{h,y}\circ\varphi_{g,x}^{-1}:\varphi_{g,x}(\U_{g,x}\cap \V_{h,y})\to \psi_{h,y}(\U_{g,x}\cap \V_{h,y})$$ is differentiable. Let $[g,u]=[h,v]\in \U_{g,x}\cap \V_{h,y}$. Then $(g^{-1},h)\in G^2$, $u\in M_{g^{-1}h}$ and $\alpha_{g^{-1}h}(u)=v$, which implies
\begin{align*}
\psi_{h,y}\circ\varphi_{g,x}^{-1}(\varphi(u))&=\psi_{h,y}\circ\varphi_{g,x}^{-1}(\varphi_{g,x}([g,u]) =\psi_{h,y}([g,u]) \\&=\psi_{h,y}([h,v])=\psi(v)=\psi(\alpha_{h^{-1}g}(u)).
\end{align*}
Thus $\psi_{h,y}\circ\varphi_{g,x}^{-1}=\psi\circ\alpha_{h^{-1}g}\circ \varphi^{-1}$ on $\varphi_{g,x}(\U_{g,x}\cap \V_{h,y}).$ It follows from the smoothness  of $\alpha_{h^{-1}g}$ that the charts $\varphi_{g,x}$ and $\psi_{h,y}$ are compatible. 
\end{proof}


In all what follows we denote by $\mathcal{A}$ the maximal atlas for the manifold $M_G$ containing all the maps $\varphi_{g,x}$ with $[g,x]\in M_G$. 
Let $x\in M$ and $(\varphi, U)$ be a chart of $M$ such that $x\in U$, it is easy to see that the representation of $\iota$ 
 in the local coordinates $\varphi$ and $\varphi_{p(x),x}$ coincides with the identity map on $U\cap M_{d(x)}$. Also, for each  $g\in G$, the representation of  $\beta_g$ 
 in the local coordinates $\varphi_{h,y}$ and $\varphi_{gh,y}$ coincides with the identity map on an open subset of $\R^n$. Thus, $\iota $ and $\beta_g$ are smooth maps. In addition, since $\iota$ and $\beta_g$ have constant rank and are bijective onto their images, we conclude that they are diffeomorphisms onto their corresponding images.

\begin{prop} \label{differentiable_t}{\rm  The following assertions hold.
\begin{enumerate}
\item[(1)] The atlas $\mathcal{A}$ is the unique differential structure on $M_G$ such  $\pi: G\times_p M\to M_G$ is a surjective submersion. 
\item[(2)] The maps $t: M_G\to G_0$ and $\beta: G\times_t M_G\to M_G$, given by~\eqref{eq_anchor} and ~\eqref{eq_globalization} respectively, are smooth.
\end{enumerate}}
\end{prop}
\begin{proof}
(1) First we show that $\pi$ is smooth with the atlas $\mathcal{A}$. For   $(g,x)\in G\times_p M$ it is enough to show that there exists an open subset $W\subseteq G\times_p M$, such that $(g,x)\in W$ and $\pi|_W$ is smooth. Let $V\subseteq G$ and $U\subseteq M$ be open sets  such that $V\times U\subseteq \Gamma_\alpha$ and $(g,x)\in V\times U$. Then $\U_{g,x}=\beta_g \iota(U\cap M_{g\m})=\{\pi(g,z)\mid z\in U\cap M_{g\m}\}$ is an open subset of $M_G$. We claim that 
\begin{equation}\label{opena}
\pi^{-1}(\U_{g,x})\cap (V\times U)=\{(h,y)\in V\times U\mid r(h)=r(g)\text{ \,\  and  \, }\,\, y\in \alpha_{h\m g}(M_{g^{-1}h}\cap U)\}.
\end{equation}
{Let us denote by $W$ the set of the right side of~\eqref{opena}. Let $(h,y)\in W$. Then, $y\in \alpha_{h\m g}(M_{g^{-1}h}\cap U)$, and as $(h,y)\in U\times V\subseteq \Gamma_\alpha$ we have that $y\in M_{h^{-1}}$. Hence $\alpha_{g^{-1}h}(y)\in \alpha_{g^{-1}h}(M_{h^{-1}}\cap \alpha_{h^{-1}g} (M_{g^{-1}h}\cap U))\subseteq U\cap M_{g^{-1}},$  and since $[h,y]=[g,\alpha_{g^{-1}h}(y)]$ we obtain $(h,y)\in\pi^{-1}(\U_{g,x})\cap (V\times U)$, which yields the inclusion $W\subseteq\pi^{-1}(\U_{g,x})\cap (V\times U)$. The other inclusion is clear.} 

 For  $(h,y)\in W={ \pi^{-1}(\U_{g,x})\cap (V\times U) }$, we have that
 \begin{equation}\label{eqq}\pi(h,y)=[h,y]=[g,\alpha_{g^{-1}h}(y)]=\beta_g([d(g),\alpha_{g^{-1}h}(y)])=\beta_g(\iota(\alpha_{g^{-1}h}(y))).\end{equation}
On the other hand,  let $L_{g^{-1}}: r^{-1}(r(g))\ni l\mapsto g\m l\in  r^{-1}(d(g))$ then by \eqref{opena}  the set $\widetilde W:=(L_{g^{-1}}\times {\rm id}_U)(W)$ is  an open subset of $G\times_p M$ contained in $\Gamma_{\alpha},$ and by \eqref{eqq} we have that \begin{equation}\label{eq-pi}
    \pi|_W=\beta_g\circ \iota\circ \alpha|_{\widetilde W}\circ(L_{g^{-1}}\times {\rm id}_U)|_W.
\end{equation} which implies that $\pi|_W$ is smooth.  
Now, to show that $\pi$ is a submersion note that $\alpha|_{\widetilde W}$ is a submersion and that the other maps on the right-hand side of \eqref{eq-pi} are diffeomorphisms. 

For the uniqueness of  $\mathcal{A}$ suppose that $\mathcal{B}$ is another differential structure of $M_G$ such that $\pi$ is a surjective submersion. Then, the identity maps  ${\rm id}:(M_G,\mathcal{A})\to (M_G,\mathcal{B})$  and ${\rm id}:(M_G,\mathcal{B})\to (M_G,\mathcal{A})$ are smooths (see \cite[Proposition 4.29]{Lee}). Thus the charts on $\mathcal{A}$ are compatible with the charts on $\mathcal{B}$ and vice versa. Therefore $\mathcal{A}=\mathcal{B}.$

(2) First of all notice that  $\tau: G\times_p M\ni (h,x)\mapsto r(h)\in G_0$ is smooth. On the other hand, since $t\circ \pi=\tau$, it follows from~\cite[Proposition 4.29]{Lee} that $t$ is smooth. Now observe that $\beta\circ({\rm id}_G\times\pi)=\pi\circ(m\times {\rm id}_{M})$, where $m: G^2\to G$ is the multiplication map. Thus,  by the fact that ${\rm id}_G\times \pi$ is a submersion and ~\cite[Proposition 4.29]{Lee} we conclude that $\beta$ is smooth.
\end{proof}

To endow $M_G$ with a structure of smooth manifold, all that remains is to prove that $M_G$ is Hausdorff. It is a well-known fact that the Hausdorff property is not  necessarily preserved by quotient spaces.  However, by \cite[Corollary 4.5]{MP1}  we can solve this by assuming that the graph of the map $\alpha: \Gamma_\alpha \to M$ is closed in $G\times M\times M.$ 

We  say that an  {\em isomorphism} between smooth partial actions of $G$ {is an isomorphism of set-theoretic partial actions} which is a diffeomorphism.


\begin{thm}\label{th2}
{\rm {Let $G$ be a star open Lie groupoid and $\alpha=(M_g, \alpha_g)_{g\in G}$ be a graph open smooth partial action of $G$ on the smooth manifold $M.$}
Then, the action $\beta: G\times_t M_G \to M_G$, given in~\eqref{eq_globalization}, is a graph open smooth globalization  if and only if
\begin{equation*}\label{gra}\gra=\{(g,x,\alpha_g(x))\in G\times M\times M\mid (g,x)\in \Gamma_\alpha \}\end{equation*}
is a closed subset of $G\times M\times M$. Moreover, in this case, this globalization is unique up to isomorphism.}
\end{thm}

\begin{proof}In the set-theoretic level, the action $\beta$ is a globalization for $\alpha,$ in particular $\iota$ and $\beta$ satisfy the conditions {(i)-(iii)} from Definition~\ref{defglobaSM}. 
By (3) of Lemma~\ref{properties},  $\iota(M)$ is open, and as $\iota$ is a diffeomorphism onto its image, we obtain (i) of Definition~\ref{defglobaSM}. Now, by item (2) of Proposition~\ref{differentiable_t} the maps $t,\beta$  are smooth,  and we have that the action $\beta$ is a globalization for $\alpha$ if and only if $M_G$ is Hausdorff. By \cite[Corollary 4.5]{MP1}, this is equivalent to that $\Gr(\alpha)$ is closed in $G\times M\times M$, showing the first assertion. To show that $\beta$ is graph open observe that from the equality $G\times_t M_G=\bigcup\limits_{e\in G_0}d^{-1}(e)\times (M_G)_e$ and  the fact that $G$ is star open, it is enough to verify that every $(M_G)_e$ is open in $M_G,$ for any  $ e\in G_0.$ For  this  note that $\pi\m ((M_G)_e)=(r^{-1}(e)\times M)\cap \M$  is open in $\M$.

Let us show the assertion about uniqueness. Suppose that $\widetilde \beta:\Gamma_{\widetilde\beta} \subseteq G\times \widetilde M_G\to \widetilde M_G$ is another globalization of $\alpha$, and let $j: M\to \widetilde M_G$ be the corresponding embedding of $M$ into $\widetilde M_G$. Let  $\phi:G\times_p M\ni (g,x)\mapsto \widetilde \beta_g(j(x))\in  M_G$. Then  $\phi$ is well-defined and  smooth, since $x\in M_{d(g)}$ implies $j(x)\in (\widetilde {M}_G)_{d(g)}$.  We claim that $\phi(g,x)=\phi(h,y)$ if and only if $[g,x]=[h,y]$. Indeed,  suppose that $\widetilde \beta_g(j(x))=\widetilde \beta_h(j(y))$. Then,  $\widetilde \beta_g(j(x))=\widetilde \beta_h(j(y)) \in (\widetilde M_G)_{r(g)}\cap(\widetilde M_G)_{r(h)}$, and the equality $\widetilde M_G=\bigsqcup\limits_{e\in G_0}(\widetilde M_G)_e$ implies $r(g)=r(h)$,  therefore
$$j(x)=\widetilde\beta_{g^{-1}h}(j(y))\in j(M)\cap (\widetilde M_G)_{g^{-1}h} \quad\mbox{ and }\quad j(y)=\widetilde\beta_{h^{-1}g}(j(y))\in j(M)\cap (\widetilde M_G)_{h^{-1}g}.$$
Thus $$j(x)=\widetilde\beta_{g^{-1}h}(j(y))\in j(M)\cap (\widetilde M_G)_{g^{-1}h} \cap \widetilde\beta_{g^{-1}h}(j(M)\cap (\widetilde M_G)_{h^{-1}g})=j(M_{g^{-1}h}),$$ yielding $x\in M_{g^{-1}h}$. Hence, 
$j(x)=\widetilde\beta_{g^{-1}h}(j(y))=j(\alpha_{g^{-1}h}(y)),$
which shows that $[g,x]=[h,y].$ Conversely, assume that  $[g,x]=[h,y].$ Then
$$j(x)=j(\alpha_{g^{-1}h}(y))=\beta_{g^{-1}h}(j(y)) \quad\mbox{and}\quad\beta_g(j(x))=\beta_h(j(y)),$$ as required.
Furthermore, $\phi$ induces an injective  map $\widetilde\phi: M_G\to \widetilde {M}_G$ such that $\widetilde\phi\circ\pi=\phi$. Since $\widetilde {M}_G$ is the $\widetilde\beta$-orbit of $j(M)$ we have that $\widetilde\phi$ is bijective, and as  $\phi$ is smooth it follows from~\cite[Proposition 4.29]{Lee} that $\widetilde\phi$ is  smooth. In addition, it is immediate to see that $\widetilde \phi$ is a morphism in ${\bf Pact}(G)$.  Since $j$ is an embedding and $\beta$ is a submersion, we have that  $\phi$ has a constant rank equal to $\dim M$. Thus, as $\pi$ is a submersion, $\widetilde\phi$ has a constant rank equal to $\dim M$. Hence $\widetilde\phi$ is a bijection with constant rank, and therefore a diffeomorphism.
\end{proof}

\begin{exe}{\rm
Let $TM=\bigsqcup\limits_{x\in M} T_xM$ be the tangent bundle of $M$. An element of $TM$ can be re\-presented as $(x,v)$, where $x\in M$ and $v\in T_xM$, and the natural projection $q:TM\ni (x,v)\mapsto x\in M$ is a submersion (see for instance ~\cite{Lee}). A  { graph open} smooth partial action $\theta=(M_g, \theta_g)_{g\in G}$ of $G$ on $M$ induces a {graph open} smooth partial action $\Theta=((TM)_g, \Theta_g)_{g\in G}$ on $TM$, where $(TM)_g=q^{-1}(M_g)$ and $\Theta_g(x,v)=(\theta_g(x),d\theta_g(x)v),$ for any $(x,v)\in (TM)_{g\m}$.  Indeed, the conditions (i) and (ii) from Definiton~\ref{ny} are clear, while condition (iii) follows from the chain rule. Further, if $p:M\to G_0$ is the anchor map of $M$, then $q\circ p$ is the anchor map of $TM$. Also, since $\theta: \Gamma_{\theta}\to M$ is smooth it follows that $\Theta: \Gamma_{\Theta}\to TM$ is smooth, which shows that $\Theta$ is a { graph open} partial action of $G$ on $TM$. 

On the other hand, if  $\phi:G\times_t N\to N$ is a globalization for $\theta$, then it is routine to prove that the action $\Phi:G\times_{\tau} TN\to TN$ induced by $\phi$, where $\tau=q\circ t$,  is a globalization for $\Theta.$ That means that the functor that sends a partial action of $G$ on $M$ to its globalization,  commutes with the functor $T$ that sends $M$ to the tangent bundle $TM$.}
\end{exe}

It is known that several topological  properties of  $M$ may not necessarily carry over to $M_G$. In the example below we will see that the property of being parallelizable does not transfer to $M_G$. For the reader's convenience we recall that a manifold $M$ of dimension $n$ is {\em parallelizable} if there exists a set $\{\sigma_1,\ldots,\sigma_n\}$ of (global) sections of $q:TM \to M$, i.e. $q\circ \sigma_i=id_M$, such that  $\{\sigma_1(x),\ldots,\sigma_n(x)\}$ is a basis of $ T_xM$ for all $x\in M,$ where we  view $T_xM$ as the fiber of $x$ under $q$. For instance, it is well known that every Lie group are parallelizable (see~\cite[Corollary 8.39]{Lee}). A well-known consequence of the Hairy Ball Theorem is that the sphere $\mathbb{S}^2$ is not parallelizable.

\begin{exe}{\rm The Mobius group ${\rm PSL}(2,\mathbb{C})$ acts smoothly on the Riemman sphere $\mathbb{S}^2=\mathbb{C}\cup\{\infty\}$, and it is  a globalization of its restriction to the complex plane $\mathbb{C}$, which is a graph open partial action of ${\rm PSL}(2,\mathbb{C})$ on $\mathbb{C}$ (see~\cite[Proposition 3.16]{KL}). The plane $\mathbb{C}$ is parallelizable, however, the sphere $\mathbb{S}^2$ is not parallelizable.}
\end{exe}

\section{On the orbit  map and orbit equivalence spaces}\label{s4}
This section is devoted to exploring the orbit map and the equivalence space arising from the partial action of a groupoid on a smooth manifold. Some of the ideas presented are inspired by classical results on multiple quotients resulting from a (global) Lie group action. 

The following result states that each orbit map {of} a smooth global action has a constant rank. As we shall see below, this observation has interesting implications for the manifold structure of the stabilizer and orbits of a smooth partial action.

\begin{prop}\label{globalrank}{\rm  Let $\beta=(M_g,\beta_g)_{g\in G}$ be a  smooth global action of  a groupoid $G$ on a smooth manifold $M,$ then for any $x\in M,$ the orbit map $\beta^x: G^x\to M,$ hasca constant rank.}
\end{prop}
\begin{proof} Let $p$   be the anchor map of $\beta,$ by \cite[Corollary 5.13]{Lee} we know that the set $G^x=\{g\in G\mid x\in M_{d(g)}\}=d^{-1}(p(x))$ is an embedded submanifold of $G$ of codimension $\dim G_0.$ Let $t: G^x\ni g\mapsto r(g)\in G_0$ and consider $\gamma: G\times_t G^x\ni (h,g)\mapsto hg\in G^x.$ It is not difficult to check that $\gamma$ is a transitive global action, with $G^x_h=\{g\in G\mid r(h)=r(g) \}.$ We shall check that $\beta^x$  is a $G$-map between $\beta$ and $\gamma,$ that is $\beta^x$ satisfies conditions (i) and (ii) in Definition \ref{mor}. For (i), take  $h\in G^x$ and $g\in G^x_h,$ then $$\beta^x(g)=\beta_g(x)\in M_{g}=M_{r(g)}=M_{r(h)}=M_{h}.$$ To check (ii) observe that for any $g\in G^x_h$ we have  $$\beta^x(\gamma_h(g))=\beta^x(hg)=\beta_{hg}(x)=\beta_h(\beta_g(x))=\beta_h(\beta^x(g)),$$ and so $\beta^x\circ \gamma_h=\beta_h\circ \beta^x$ on $G^x_h.$ Now, to show that $\beta^x$ has a constant rank, take $g_1,g_2\in G^x.$ Since $\gamma$ is transitive, there exists $k\in G$ such that $\gamma_k(g_1)=g_2$,  and thus $\beta^x(g_2)=\beta^x\circ \gamma_k(g_1)=\beta_k\circ \beta^x (g_1).$ This implies that the  following diagram commutes
\[\begin{tikzcd}
T_{g_1}G^x \arrow{r}{d\beta^x(g_1)} \arrow[swap]{d}{d\gamma_k(g_1)} & T_{\beta^x(g_1)}G^x\arrow{d}{d\beta_k(\beta^x(g_1))} \\
T_{g_2}G^x \arrow{r}{d\beta^x(g_2)} & T_{\beta^x(g_2)}G^x
\end{tikzcd}.
\]
As $\gamma_k$ and $\beta_k$ are diffeomorphisms we conclude that the rank of  $\beta^x$ at the points $g_1$ and $g_2$ coincide, which implies that $\beta^x$ has a constant rank, as desired.
\end{proof}

\begin{cor}\label{cor_constantrank}{\rm Let $\alpha=(\alpha_g, M_g)_{g\in G}$  be a globalizable graph open smooth partial action of $G$ on $M$. Then,  $G$ is star open  and the orbit map $\alpha^x:G^x\to M$ has a constant rank for any $x\in M$.}
\end{cor}
\begin{proof}
Let $x\in M$ and  $\beta=(N_g,\beta_g)$ be a globalization for $\alpha$.  Let $G_\alpha^x=\{g\in G\mid x\in M_{g^{-1}}\}$ and $G_{\beta}^x=\{g\in G\mid x\in N_{d(g)}\}=d^{-1}(p(x))$, we have that $\alpha^x: G^x_\alpha \to M$ and $\beta^x: G^x_\beta\to N$ are orbit maps of $\alpha$ and $\beta$, repectively. By~\cite[Corollary 5.13]{Lee}, $G_\beta^x$ is an embedded submanifold of $G$ of codimension equal to $\dim G_0$, and as $G_\alpha^x$ is an open subset of $G$ contained in $G_\beta^x$, we have that $ \dim G^x_\beta=\dim G$. So, { by Proposition~\ref{soc}} $\dim G_0=0$ and the first assertion follows. Now, $\alpha^x$ has a constant rank since it is the restriction to $\beta ^x$ to an open set,  and $\beta ^x$ has a constant rank thanks to Proposition~\ref{globalrank}, this finishes the proof.
\end{proof}

The following result is a straightforward observation concerning the structure of the stabilizer and the orbit of a smooth partial action.

\begin{cor}\label{lemma_orbit_stab}{\rm  Let $G$ be a Lie groupoid, {$M$ be a smooth manifold, and $\alpha=(\alpha_g, M_g)_{g\in G}$ be a graph open smooth partial action of $G$ on $M$} and let $x\in M$.  Then  ${\rm 
stab}_\alpha(x)$ is an embedded Lie subgroup of $G_{p(x)}$ provided  that {the} orbit map $\alpha^x: G^x\to M$ has  a constant rank (for instance if $\alpha$ has a smooth globalization). Moreover, if $\alpha$ is also free, then the orbit $\alpha^x(G^x)$ of $x$ is an immersed submanifold of $M$.}
\end{cor}

\begin{proof} Suppose first that  the   orbit map $\alpha^x: G^x\to M,$ has a constant rank.
From \cite[Theorem 5.12]{Lee} it follows that ${\rm stab}_\alpha(x)=(\alpha^x)^{-1}(x)$ is an embedded submanifold of the Lie group $G_{p(x)}$, and thus a Lie subgroup of $G_{p(x)}$ by~\cite[Theorem 7.11]{Lee}.  For the last assertion, note that if $\alpha$ is also free, then $\alpha^x$ is injective and thus an immersion. Hence $\alpha^x(G^x)$ is an immersed submanifold of $M$ thanks to \cite[Proposition 5.18]{Lee}.
\end{proof}




\begin{rem}\label{remark_transitive} {\rm Let $\overline{G}^x=\{\A\in \overline H^x\mid x\in X_{h^{-1}}\},$ where $\overline{H}^x$ is defined in \eqref{barhx}. We have from  Theorem~\ref{th1} that the map $\varphi|_{\overline{G}^x}:\overline{G}^x\to X$ is an isomorphism between the restriction of $\delta$ to $\overline{G}^x$ and $\alpha$. In the case that $\alpha$ is free and transitive, the restriction $\varphi|_{\overline{G}^x}$ coincides with the orbit map $\alpha^x: G^x\to X$, and  the action $\delta: G\times _t H^x\ni (g,h) \mapsto gh\in H^x$ is a globalization for $\alpha$ with corresponding  injection $(\alpha^x)^{-1}: X\to G^x\subseteq H^x$. With this in mind, we can address the problem of globalization for free and transitive smooth partial actions.}
\end{rem}

\begin{thm}\label{thr1_smooth}{\rm 
Let $\alpha=(M_g, \alpha_g )_{g\in G}$ be a free and transitive  graph open smooth partial action of a Lie groupoid $G$ on a smooth manifold $M$. Then the following assertions are equivalent:
\begin{enumerate}
\item[(1)] The orbit map $\alpha^x:G^x\to X$ has a  constant rank, for  all $x\in M$.
\item[(2)] The orbit map $\alpha^x:G^x\to X$ has a constant rank , for  some $x\in M$.
\item[(3)] $\alpha$ is globalizable.
\end{enumerate}}
\end{thm}
\begin{proof}
The implication (1) $\Rightarrow$ (2) is immediate, and (3) $\Rightarrow$ (1) follows from Corollary~\ref{cor_constantrank}. It remains to show  (2) $\Rightarrow$ (3). Suppose that for some $x\in M$ we have that $\alpha^x: G^x\to X$ has a constant rank. As $d:G\to G_0$ is a submersion, we have by \cite[Corollary 5.13]{Lee} that  $H^x=d^{-1}(p(x))$ is an embedded submanifold of $G$. {It is easy to see that $\overline{H}^x=H^x$ and that the global action $\delta$ of $G$ on $H^x$ is smooth (see (6)). As observed in Remark \ref{remark_transitive}  the action $\delta$}  is  a set- theoretic globalization of $\alpha$ with corresponding injection $(\alpha^x)^{-1}: X\to H^x$. It only remains to verify that $(\alpha^x)^{-1}$ is a diffeomorphism onto an open subset of $H^x$.  As $\alpha$ is graph open, then $G^x$ is open in $G$, and by assumption $\alpha^x: G^x\to X$ has a constant rank, which implies  that $(\alpha^x)^{-1}: X\to H^x$ is a diffeomorphism onto the open set $G^x$. Therefore $\delta$ is a globalization of $\alpha$, as required.
 
For the final assertion, notice that he action $\delta: G\times_t H^x\ni (g,h) \mapsto gh\in H^x$, defined in \eqref{delta}, is a globalization for $\alpha$, where $H^x=d^{-1}(p(x)).$ Suppose that a smooth action $\beta=(N_g,\beta_g)$ of $G$ on $N$ is another globalization for $\alpha.$  By Theorem~\ref{th1}, the orbit map $\beta^x: H^x\to N$ is a set-theoretic $G$-equivalence between $\delta$ and $\beta$. Now, as $\beta^x$ is a smooth bijection with constant rank (by Proposition~\ref{globalrank}), we conclude that $\beta^x$ is a diffeomorphism, as required. 
\end{proof}




Let $G$ be a Lie groupoid, $M$ be a smooth manifold, and $\alpha=(\alpha_g, M_g)_{g\in G}$ be a graph open smooth partial action of $G$ on $M.$ 
We denote by $M/G$ the {\it orbit equivalence space} of $M,$ that is the elements of $M/G$ are the orbits of $\alpha,$ namely the sets $G^x\cdot x$ defined in \eqref{or}. The space  $M/G$ is endowed with the quotient topology, in particular, the quotient map $\pi_G: M\to M/G$ is continuous. Also,  for any open subset $U$ of $M,$ we have that$$\pi\m_G(\pi_G(U))=\bigcup_{g\in G}\alpha_g(U\cap M_{g\m}),$$ which implies that  $\pi_G$ is an open map.  {We shall say that} a smooth partial action $\alpha=(M_g, \alpha_g)_{g\in G}$ of $G$ on $M$ is called  \textit{proper,} if the map ${\Lambda}: \Gamma_\alpha \to M\times M$ given by $(g,x)\mapsto (x, \alpha_g(x))$ is proper.

\begin{thm}\label{th3}
{\rm Let $G$ be a Lie groupoid, $M$ be a smooth manifold, and $\alpha=(\alpha_g, M_g)_{g\in G}$ be  a graph open smooth partial action of $G$ on $M.$ Suppose that $\alpha$ is free, proper, and for every $x\in M$ the orbit map $\alpha^{x}: G^x\to M$ has a constant rank. Then, there exists exactly one differentiable structure on $M/G$ such that the quotient map $\pi_G: M\to M/G$ is a submersion.
} 
\end{thm}
\begin{proof} 
{ It is enough to verify that $\Lambda(\Gamma_\alpha)=\{(x,\alpha_g(x))\mid (g,x)\in \Gamma_\alpha\}$ is a proper immersed submanifold of $M\times M,$ and that the first projection ${\rm pr}_1: \Lambda(\Gamma_\alpha)\to M$   is a submersion thanks to Godement's criterium (see~\cite[Theorem 2, p. 92]{Serre})}. As $\Lambda$ is proper, $\Lambda(\Gamma_\alpha)$ is closed in $M\times M$ {(see~\cite[Theorem A.57]{Lee})}. So, to guarantee that $\Lambda(\Gamma_\alpha)$ is a proper immersed submanifold of $M\times M$ it is enough to show that $\Lambda$ is an immersion. Note that $\Lambda$ is smooth and  by  \eqref{edife} we have
$$d\Lambda(g,x)(v,w)=(w,d\alpha^{x}(g)v+d\alpha_{g}(x)w), $$
for each $(v,w)\in T_{g}G\times T_{x}M$. Then, 
$\Lambda$ is an immersion if and only if $\alpha^x$ is an immersion for every $x\in M$. Since $\alpha$ is free, for all $x\in M$  we have that $\alpha^x$ is injective, and therefore an immersion, as it has a constant rank, as required. Finally,  the first projection ${\rm pr}_1: \Lambda(\Gamma_\alpha)\to M$ is  a submersion, which finishes the proof.
\end{proof}

Our last result in this work is the following  Corollary, which is a consequence of Proposition~\ref{globalrank} and Theorem~\ref{th3}.
\begin{cor}\label{cf}
{\rm Let $G$ be a Lie groupoid, $M$ a smooth manifold, and {$\alpha=(\alpha_g, M_g)_{g\in G}$ be a graph open smooth partial action of $G$ on $M$, free and proper.} If $\alpha$ has a smooth globalization, then there exists exactly one differentiable structure on $M/G$ such that the quotient map $\pi_G: M\to M/G$ is a submersion.}
\end{cor}

\noindent{\bf Declarations:}\vspace{0.5cm}

 \noindent The authors have no conflict of interest to declare that are relevant to this article.


\end{document}